\documentclass[reqno,a4paper]{amsart}

\usepackage{tikz}
\usepackage{subcaption}

\usepackage{hyperref}
\usetikzlibrary{math}
\usepackage[all,cmtip]{xy}
\definecolor{awesome}{rgb}{1.0, 0.13, 0.32}
\definecolor{bondiblue}{rgb}{0.0, 0.58, 0.71}
\definecolor{cocoabrown}{rgb}{0.82, 0.41, 0.12}
\definecolor{ufogreen}{rgb}{0.24, 0.82, 0.44}
\definecolor{bananamania}{rgb}{0.98, 0.91, 0.71}

\newcommand{\balpha}{{\boldsymbol{\alpha}}}
\newcommand{\Z}{\mathbb{Z}}
\newcommand{\R}{\mathbb{R}}
\newcommand{\C}{\mathbb{C}}

\title[Complex continued fraction algorithms]{A finiteness condition for complex continued fraction algorithms}

\author[C.~Kalle]{Charlene~Kalle}
\author[F.~M.~S\'elley]{Fanni~M.~S\'elley}
\author[J.~M.~Thuswaldner]{J\"org~M.~Thuswaldner}
\address[C.K.\ and F.S.]{Mathematisch Instituut, Leiden University, Niels Bohrweg 1, 2333CA Leiden, The Netherlands}
\email{kallecccj@math.leidenuniv.nl}
\email{selleyf@gmail.com}
\address[J.T.]{Lehrstuhl f\"ur Mathematik und Statistik, Montanuniversit\"at Leoben, Franz Josef Stra\ss{}e~18, A-8700 Leoben, Austria.}
\email{joerg.thuswaldner@unileoben.ac.at}

\begin{document}

\newtheorem{prop}{Proposition}[section]
\newtheorem{theorem}{Theorem}[section]
\newtheorem{lemma}{Lemma}[section]
\newtheorem{cor}{Corollary}[section]
\newtheorem{remark}{Remark}[section]
\theoremstyle{definition}
\newtheorem{defn}{Definition}[section]
\newtheorem{ex}{Example}[section]


\begin{abstract}
It is desirable that a given continued fraction algorithm is simple in the sense that the possible representations can be characterized in an easy way. In this context the so-called {\it finite range condition} plays a prominent role. We show that this condition holds for complex {\it $\balpha$-Hurwitz algorithms} with parameters $\balpha\in\mathbb{Q}^2$. This is equivalent to the existence of certain finite partitions related to these algorithms and lies at the root of explorations into their Diophantine properties. Our result provides a partial answer to a recent question formulated by Lukyanenko and Vandehey.
\end{abstract}

\maketitle

\section{Introduction}
Since antiquity, many different kinds of representations of numbers by strings of symbols have been used for various purposes. The most famous one, the {\it decimal system}, has a particularly nice property: We can represent positive integers by forming finite strings of digits taken from the set $\{0,\ldots, 9\}$ in an arbitrary way. If we forbid leading zeros, each of these strings represents a positive integer uniquely. 

Things are not always that easy and often we get restrictions on the digit strings. For instance, consider the {\it Fibonacci sequence} $(F_n)_{n\ge 0}$ defined by $F_0=0$, $F_1=1$, and $F_{n}=F_{n-1}+ F_{n-2}$ for $n \ge 2$. Each positive integer $N\in\mathbb{N}$ can be represented in the form
\begin{equation}\label{eq:Zeckendorf}
N= \sum_{j=0}^{L-1} \varepsilon_j F_{j+2}
\end{equation}
(for some ``length'' $L\in \mathbb{N}$) with digits $\varepsilon_0,\ldots,\varepsilon_{L-1} \in \{0,1\}$.
However, even if we forbid leading zeros, uniqueness is maintained only if the digit string $\varepsilon_0\cdots\varepsilon_{L-1}$ does not contain two consecutive ones, {\it i.e.}, if the pattern $11$ is forbidden. In this case the expansion \eqref{eq:Zeckendorf} is called the
{\it Zeckendorf expansion} of $N$. Expansions of this kind go back to \cite{Z:72}.  

In the present paper, we are concerned with {\it continued fraction algorithms}. They allow to represent a real or complex number $z$ by digits (called {\it partial quotients}) $a_1,a_2,a_3,\ldots$ taken from an infinite set (like, {\it e.g.}, the set of positive integers) as
\begin{equation}\label{eq:cfintro}
z = \cfrac1{a_1+\cfrac1{a_2 +\cfrac1{a_3 + \ddots}}}.
\end{equation}
The {\it classical continued fraction algorithm} for real numbers is a generalization of {\it Euclid’s algorithm} which performs division with remainder on two integers. 
It can be described dynamically 
through the {\it Gauss map} 
\[
G:(0,1) \to[0,1);\quad z \mapsto \frac1z-\Big\lfloor\frac1z\Big\rfloor,
\]
where $\lfloor x\rfloor = \max\{n\in \mathbb{Z} \,:\, n\le x\}$ denotes the integer part of a real number $x$. 
For $z \in \mathbb R$ we iteratively apply $G$ on $z_0 = z-\lfloor z\rfloor$. Writing $z_n=G^n(z_0)$ and $a_n=\lfloor\tfrac{1}{G^{n-1}(z_0)}\rfloor$ for each $n \ge 1$ we find that
\[ z_0=
\frac1{a_1 + z_1}=  \cfrac1{a_1 + \cfrac1{a_2+z_2}} = \cdots =  \cfrac{1}{a_1 + \cfrac{1}{a_2 + \cfrac1{\ddots +\cfrac{1}{a_n+z_n}}}}, \]
provided that $z_0,\ldots, z_{n-1}$ do not vanish. If $z_0$ is irrational, $z_n$ will never vanish and the process can be performed {\it ad infinitum} leading to the infinite representation~\eqref{eq:cfintro} of $z_0=z-\lfloor z \rfloor$ (see \cite[Chapter~I]{RS:92}). In fact, one can show that each infinite string $a_1a_2a_3\cdots$ of positive integers uniquely represents an irrational number $z_0\in (0,1)$ in the form \eqref{eq:cfintro}, so there are no forbidden strings of digits. The sequence of integers $(a_i)_{i\ge 1}$ is called  the {\it continued fraction expansion}\footnote{The term ``continued fraction'' was introduced by John Wallis in 1656 in his {\it arithmetica infinitorum}~\cite{wallis1656}.} of $z_0$ and~\eqref{eq:cfintro} is written shortly as $z_0=[a_1,a_2,\ldots]$.

If we consider only the first $n$ digits of the expansion \eqref{eq:cfintro} we obtain the rational numbers 
\[
\cfrac{p_n}{q_n}=[a_1,a_2,\ldots, a_n]= \cfrac{1}{a_1 + \cfrac{1}{a_2 + \cfrac1{\ddots +\cfrac{1}{a_n}}}},
\]
the {\it convergents} of $z_0$. These convergents surface naturally in Diophantine approximation, since by Lagrange's best approximation law they optimally approximate $z_0$ in the sense that
$
| \frac{p_n}{q_n} - z_0 | < | \frac{p}{q} - z_0 |
$
holds for all $p,q\in\mathbb{N}$ with $q\le q_n$ and $(p,q)\not=(p_n,q_n)$.

Real numbers can be represented by continued fraction expansions in many other ways. Arguably the most efficient continued fraction algorithm is the {\em nearest integer algorithm} introduced first by Bernhard Minnigerode~\cite{Min73} in 1873. For each irrational number $z\in[-\tfrac12, \tfrac12)\setminus\{0\}$ this algorithm produces an expansion of the form \eqref{eq:cfintro} with $a_n\in\mathbb{Z}$, $|a_n| \ge 2$, for all $n \ge 1$. Nearest integer continued fraction expansions can be obtained dynamically in a way very similar to above. In particular, letting  
\[ 
F: \Big[-\frac12, \frac12 \Big)\setminus\{0\} \to \Big[ -\frac12, \frac12 \Big);\quad 
z \mapsto \frac 1z - \Big\lfloor \frac1z \Big\rfloor_{1/2},
\]
where $\lfloor z \rfloor_{1/2}$ denotes the unique integer $a$ satisfying $z-a \in \left[ -\frac12, \frac12 \right)$, one sets $a_n=\lfloor\tfrac{1}{F^{n-1}(z_0)}\rfloor_{1/2}$ for each $n \ge 1$. The efficiency of this algorithm lies in the fact that if one considers the rational approximations to $z$ given by the convergents $\frac{p_n}{q_n}$ of this algorithm, then typically the sequence $(\frac{p_n}{q_n})_{n \ge 1}$ converges faster to $z$ than the convergents of other similar algorithms. However, like for the Fibonacci expansion, we need restrictions on the possible digit strings in order to get unique expansions.  

For the algorithms $G$ and $F$, write $\Delta_G(a_1, \ldots, a_n)$ and $\Delta_F(a_1, \ldots, a_n)$ for the {\em cylinder set} of all $x$ that have $a_1, \ldots, a_n$ as the first $n$ digits in their expansion according to $G$ and $F$, respectively. The fact that the set of forbidden strings of digits is easy to describe relates to the {\em finite range property} as given in \cite{Yur95}. More specifically, the collections
\[ \begin{split}
\mathcal G = \, & \{ G^n (\Delta_G(a_1, \ldots, a_n))\, : \, a_1, \ldots, a_n \in \mathbb N,\, n \ge 1\},\\
\mathcal F =\, &\{ F^n (\Delta_F(a_1, \ldots, a_n))\, :  \, a_1, \ldots, a_n \in \mathbb N,\, n \ge 1\}
\end{split}\]
both contain only finitely many sets up to Lebesgue null sets. In fact, we have $\mathcal G = \{ [0,1)\}$ and $\mathcal F = \{ [ -\frac12, 0), (0, \frac12]\}$. Note that $\mathcal G$ is a singleton, which reflects the fact that each string $a_1a_2\cdots \in \mathbb{N}^\mathbb{N}$ can occur as a classical continued fraction expansion. On the contrary, $\mathcal F$ contains two elements which entails that the nearest integer continued fraction expansions are subject to restrictions on the digit strings. An automata theoretic description of such restrictions for a wide class of continued fraction algorithms for real numbers, including Minnegerode's algorithm, is presented in \cite{Sha13}. For general information on continued fraction algorithms see for instance the classical work of Oskar Perron~\cite{Per13}.

An example of a continued fraction algorithm for complex numbers is furnished by the {\em Hurwitz algorithm} that will be presented in Section~\ref{sec:hurwitz}. Its digit strings $a_1a_2a_3\cdots$ are characterized by a finite range condition that is governed by the finite partition depicted in Figure~\ref{fig:markov1/2,1/2}
\begin{figure}[h]
         \centering
    \includegraphics[width=0.3\textwidth]{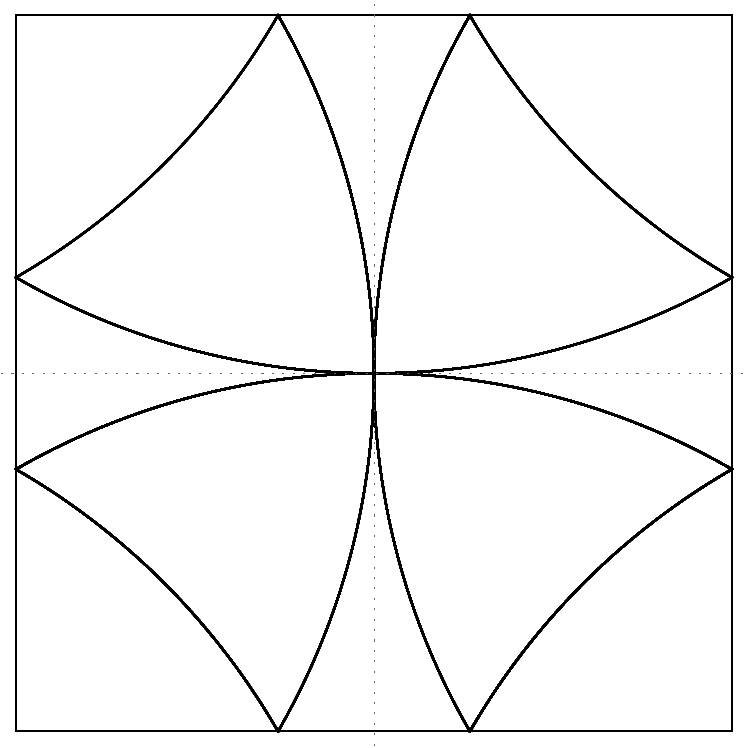}
         \caption{A finite partition for the Hurwitz algorithm. 
         } \label{fig:markov1/2,1/2}
\end{figure}
(see Section~\ref{sec:hurwitz} for further explanations). The Hurwitz algorithm is studied in the literature with a strong revival recently, see {\it e.g.}~\cite{Abr22,BGH23,hiary2022calculations,LV23} and the references therein. For a historical account we refer to \cite{oswald2014complex}.  The finite range condition lies at the heart of the results in \cite{Ei2018,lakein1973approximation,Nak76} that give estimates on the approximations $|z-\frac{p_n}{q_n}|$, $n \ge 1$, see also \cite{Hensley:06}.

One-parameter perturbations of the Hurwitz algorithm, so called {\em $\balpha$-Hurwitz algorithms}, gained interest recently ({\em cf}.~for instance \cite{Abr22,dani2014continued,LV23,LV24}). We introduce them in Section~\ref{sec:alpha}. It has been asked if the set of possible digit strings $a_1a_2a_3\cdots$ produced by the $\balpha$-Hurwitz algorithm for a given parameter $\balpha$ can be characterized by a finite range condition, and it is our aim to shed some light on this problem. In particular, we will show that for each rational parameter $\balpha$, the $\balpha$-Hurwitz algorithm satisfies this condition by constructing a finite partition for them (this result is stated in Theorem~\ref{t:mp}). In doing so, we give a partial answer to a question stated in \cite[Section~1]{LV24}. Moreover, since the finite range condition is crucial when one wants to derive further Diophantine properties of a continued fraction algorithm, Theorem~\ref{t:mp} can be regarded as a starting point for the further exploration of $\balpha$-Hurwitz algorithms. The proof of our result is contained in Section~\ref{sec:markov}, and we give some perspectives for further research in Section~\ref{sec:perspectives}.

\section{The Hurwitz algorithm}\label{sec:hurwitz}
At the end of the nineteenth century the two brothers Adolf and Julius Hurwitz both created continued fraction algorithms for complex numbers, see \cite{H:1887,HurJ02}. Here we consider the complex continued fraction algorithm introduced by Adolf Hurwitz \cite{H:1887} in 1887, which is a generalization to the complex plane of the nearest integer algorithm for real numbers. Let
$
\mathbb Z[i] = \{ a+bi \, : \, a, b \in \mathbb Z\}
$
with $i=\sqrt{-1}$ denote the set of {\it Gaussian integers} and set
\[ 
U:=\big\{ z \in \mathbb{C} \, : \, -\tfrac12 \le \Re z, \Im z < \tfrac12\big\},
\]
where $\Re z$ and $\Im z$ denote the real and imaginary part of $z\in \mathbb{C}$, respectively. The {\it Hurwitz algorithm} is defined in total analogy to the nearest integer continued fraction algorithm. The set $U$ will play the role of the interval $[-\frac12, \frac12)$ and the floor function $\lfloor \cdot \rfloor_{1/2}$ is replaced by the function $\lfloor \cdot \rfloor_U : \C \to U;\,  z  \mapsto z-w$, where $w$ is the unique Gaussian integer satisfying $z - w \in U$. The {\it Hurwitz map} is given by 
\begin{equation}\label{eq:mapT}
T:U\setminus\{0\} \to U;\quad z \mapsto \frac1z - \Big\lfloor \frac1z \Big\rfloor_U, 
\end{equation}
and the Hurwitz algorithm is defined in terms of iterates of $T$. In particular, using this map one can assign to each $z \in U$ a {\it Hurwitz expansion} of the form \eqref{eq:cfintro}. Here $a_1,a_2,\ldots \in \mathbb Z[i] \setminus \{ 0, \pm 1, \pm i\}$ are given by $a_n=\lfloor \frac1{T^{n-1}(z)}\rfloor$ provided that $T^{n}(z)$, $n\ge 0$, do not vanish. If one of these iterates vanishes, the expansion of $z$ becomes finite. In any case, the Hurwitz expansion of $z\in U$ represents $z$ (see {\it e.g.}~\cite[Chapter~5]{Hensley:06} for a detailed account on this algorithm). 
Let us define {\it cylinder sets} recursively as follows. For each $b \in \mathbb Z[i]$ set 
$
\Delta(b) = \big\{ z \in U \, : \, \left\lfloor \frac1z \right\rfloor_U = b\big\}
$
and let $T_{b}$ be the restriction of $T$ to $\Delta(b)$. Then, for $b_1,\ldots, b_n \in \mathbb Z[i]$, $n\ge 2$, we define 
\[
\Delta(b_1,\ldots, b_n) = T_{b_1}^{-1}(\Delta(b_2,\ldots, b_n)) \cap U.
\]
This entails that the {\em cylinder set} $\Delta(b_1,\ldots, b_n)$ is the set of points $z \in U$ that share the same first digits $a_1=b_1, \ldots, a_n=b_n$ in their Hurwitz expansion~\eqref{eq:cfintro}.
\begin{figure}[h]
\centering
\begin{subfigure}[b]{0.47\textwidth}
\centering
\begin{tikzpicture}[scale=6.2]
\tikzmath{\a1 = 0.5; \a2 = 0.5;}

\foreach \y in {2,3,...,15} {
\draw [smooth, black!50,thick, domain=0:360] plot ({1/(2*(\a1+\y))+1/(2*(\a1+\y) )*cos(\x)}, {1/(2*(\a1+\y) )*sin(\x)});
\draw [smooth, black!50,thick, domain=0:360] plot ({1/(2*(\a2 +\y))*cos(\x)}, {-1/(2*(\a2+\y))+1/(2*(\a2+\y) )*sin(\x)});
}
\draw [smooth, black!50,thick, domain=150:210] plot ({1/(2*\a1)+1/(2*\a1)*cos(\x)}, {1/(2*\a1) )*sin(\x)});
\draw [smooth, black!50,thick, domain=60:300] plot ({1/(2*(\a1+1))+1/(2*(\a1+1) )*cos(\x)}, {1/(2*(\a1+1) )*sin(\x)});

\draw [smooth, black!50,thick, domain=240:300] plot ({1/(2*(1-\a2 ))*cos(\x)}, {1/(2*(1-\a2))+1/(2*(1-\a2) )*sin(\x)});
\draw [smooth, black!50,thick, domain=150:390] plot ({1/(2*(2-\a2 ))*cos(\x)}, {1/(2*(2-\a2))+1/(2*(2-\a2) )*sin(\x)});
\foreach \y in {3,4,...,16} {
\draw [smooth, black!50,thick, domain=0:360] plot ({1/(2*(\y-\a2 ))*cos(\x)}, {1/(2*(\y-\a2))+1/(2*(\y-\a2) )*sin(\x)});
\draw [smooth, black!50,thick, domain=0:360] plot ({1/(2*(\a1-\y))+1/(2*(\y-\a1) )*cos(\x)}, {1/(2*(\y-\a1) )*sin(\x)});
}

\draw [smooth, black!50,thick, domain=330:390] plot ({1/(2*(\a1-1))+1/(2*(1-\a1) )*cos(\x)}, {1/(2*(1-\a1) )*sin(\x)});
\draw [smooth, black!50,thick, domain=240:480] plot ({1/(2*(\a1-2))+1/(2*(2-\a1) )*cos(\x)}, {1/(2*(2-\a1) )*sin(\x)});

\draw [smooth, black!50,thick, domain=60:120] plot ({1/(2*\a2 )*cos(\x)}, {-1/(2*\a2)+1/(2*\a2 )*sin(\x)});
\draw [smooth, black!50,thick, domain=570:330] plot ({1/(2*(\a2 +1))*cos(\x)}, {-1/(2*(\a2+1))+1/(2*(\a2+1) )*sin(\x)});

\filldraw[fill=black!50, draw=black!50] (\a1-1/2, \a2-1/2) circle (1.5pt);

\draw[thick,black!50!white] (\a1-1, \a2-1) -- (\a1-1, \a2) -- (\a1, \a2) -- (\a1, \a2 -1) -- cycle;

\node[font=\small, scale=0.7] at (\a1-.05,\a2-1/2) {\color{black!50} 2};
\node[font=\small, scale=0.7] at (\a1-.15,\a2-1/2) { \color{black!50} 3};
\node[font=\small, scale=0.7] at (\a1-1+.15,\a2-1/2) { \color{black!50} $-3$};
\node[font=\small, scale=0.7] at (\a1-.25,\a2-1/2) { \color{black!50} 4};
\node[font=\small, scale=0.7] at (\a1-1+.05,\a2-1/2) { \color{black!50} $-2$};
\node[font=\small, scale=0.7] at (\a1-1/2,\a2-.05) { \color{black!50} $-2i$};
\node[font=\small, scale=0.7] at (\a1-1/2,\a2-.15) { \color{black!50} $-3i$};
\node[font=\small, scale=0.7] at (\a1-1/2,\a2-1+.05) { \color{black!50} $2i$};
\node[font=\small, scale=0.7] at (\a1-1/2,\a2-1+.15) { \color{black!50} $3i$};
\node[font=\small, scale=0.7] at (.41,.43) { \color{black!50} $1-i$};
\node[font=\small, scale=0.7] at (.4,.22) { \color{black!50} $2-i$};
\node[font=\small, scale=0.7] at (.2,.43) { \color{black!50} $1-2i$};
\node[font=\small, scale=0.7] at (.25,.25) { \color{black!50} $2-2i$};

\node[font=\small, scale=0.7] at (.41,-.43) { \color{black!50} $1+i$};
\node[font=\small, scale=0.7] at (.4,-.22) { \color{black!50} $2+i$};
\node[font=\small, scale=0.7] at (.2,-.43) { \color{black!50} $1+2i$};
\node[font=\small, scale=0.7] at (.25,-.25) { \color{black!50} $2+2i$};

\node[font=\small, scale=0.7] at (-.41,-.43) { \color{black!50} $-1+i$};
\node[font=\small, scale=0.7] at (-.4,-.22) { \color{black!50} $-2+i$};
\node[font=\small, scale=0.7] at (-.2,-.43) { \color{black!50} $-1+2i$};
\node[font=\small, scale=0.7] at (-.26,-.25) { \color{black!50} $-2+2i$};

\node[font=\small, scale=0.7] at (-.41,.43) { \color{black!50} $-1-i$};
\node[font=\small, scale=0.7] at (-.4,.22) { \color{black!50} $-2-i$};
\node[font=\small, scale=0.7] at (-.2,.43) { \color{black!50} $-1-2i$};
\node[font=\small, scale=0.7] at (-.26,.25) { \color{black!50} $-2-2i$};
\end{tikzpicture}
\caption{The set $U$. The regions bounded by the circle arcs indicate the cylinder sets $[b]$ for the digits $b \in \mathbb Z[i] \setminus \{0, \pm 1, \pm i\}$. Note that we have $\tfrac1{\Delta(b)} = (U + b) \cap \frac1U$. \\[2pt]}
\end{subfigure}
\hfill
\begin{subfigure}[b]{0.48\textwidth}
         \centering
\begin{tikzpicture}[scale=.88]
\tikzmath{\a1 = 0.5; \a2 = 0.5;}

\draw[smooth, black!50!white] (\a1-4, \a2-4) -- (\a1-4, \a2+3);
\draw[smooth, black!50!white] (\a1-3, \a2-4) -- (\a1-3, \a2+3);
\draw[smooth, black!50!white] (\a1-2, \a2-4) -- (\a1-2, \a2+3);
\draw[smooth, black!50!white] (\a1-1, \a2-4) -- (\a1-1, \a2+3);
\draw[smooth, black!50!white] (\a1, \a2-4) -- (\a1, \a2+3);
\draw[smooth, black!50!white] (\a1+1, \a2-4) -- (\a1+1, \a2+3);
\draw[smooth, black!50!white] (\a1+2, \a2-4) -- (\a1+2, \a2+3);
\draw[smooth, black!50!white] (\a1+3, \a2-4) -- (\a1+3, \a2+3);

\draw[smooth, black!50!white] (\a1-4, \a2-4) -- (\a1+3, \a2-4);
\draw[smooth, black!50!white] (\a1-4, \a2-3) -- (\a1+3, \a2-3);
\draw[smooth, black!50!white] (\a1-4, \a2-2) -- (\a1+3, \a2-2);
\draw[smooth, black!50!white] (\a1-4, \a2-1) -- (\a1+3, \a2-1);
\draw[smooth, black!50!white] (\a1-4, \a2) -- (\a1+3, \a2);
\draw[smooth, black!50!white] (\a1-4, \a2+1) -- (\a1+3, \a2+1);
\draw[smooth, black!50!white] (\a1-4, \a2+2) -- (\a1+3, \a2+2);
\draw[smooth, black!50!white] (\a1-4, \a2+3) -- (\a1+3, \a2+3);

\draw [awesome,very thick, domain=270:450] plot ({1/(2*\a1)+1/(2*\a1 )*cos(\x)}, {1/(2*\a1 )*sin(\x)});
\draw[awesome,very thick] (\a1,\a2-1)--(\a1,\a2);
\draw [bondiblue,very thick, domain=0:180] plot ({1/(2*(1-\a2 ))*cos(\x)}, {1/(2*(1-\a2))+1/(2*(1-\a2) )*sin(\x)});
\draw[ufogreen, very thick] (\a1, \a2) -- (\a1-1,\a2);
\draw [cocoabrown,very thick, domain=90:270] plot ({1/(2*(\a1-1))+1/(2*(1-\a1) )*cos(\x)}, {1/(2*(1-\a1) )*sin(\x)});
\draw [cocoabrown,very thick] (\a1-1,\a2)--(\a1-1,\a2-1);
\draw [ufogreen,very thick, domain=180:360] plot ({1/(2*\a2 )*cos(\x)}, {-1/(2*\a2)+1/(2*\a2 )*sin(\x)});
\draw[bondiblue,very thick] (\a1-1,\a2-1)--(\a1,\a2-1);

\node[font=\small, scale=0.7] at (\a1-7/2,\a2-1/2) {\color{black!50!white} $-3$};
\node[font=\small, scale=0.7] at (\a1-5/2,\a2-1/2) {\color{black!50!white} $-2$};
\node[font=\small, scale=0.7] at (\a1-3/2,\a2-1/2) {\color{black!50!white} $-1$};
\node[font=\small, scale=0.7] at (\a1-1/2,\a2-1/2) {\color{black!50!white} $0$};
\node[font=\small, scale=0.7] at (\a1+1/2,\a2-1/2) {\color{black!50!white} $1$};
\node[font=\small, scale=0.7] at (\a1+3/2,\a2-1/2) {\color{black!50!white} $2$};
\node[font=\small, scale=0.7] at (\a1+5/2,\a2-1/2) {\color{black!50!white} $3$};

\node[font=\small, scale=0.7] at (\a1-7/2,\a2+1/2) {\color{black!50!white} $-3+i$};
\node[font=\small, scale=0.7] at (\a1-5/2,\a2+1/2) {\color{black!50!white} $-2+i$};
\node[font=\small, scale=0.7] at (\a1-3/2,\a2+1/2) {\color{black!50!white} $-1+i$};
\node[font=\small, scale=0.7] at (\a1-1/2,\a2+1/2) {\color{black!50!white} $i$};
\node[font=\small, scale=0.7] at (\a1+1/2,\a2+1/2) {\color{black!50!white} $1+i$};
\node[font=\small, scale=0.7] at (\a1+3/2,\a2+1/2) {\color{black!50!white} $2+i$};
\node[font=\small, scale=0.7] at (\a1+5/2,\a2+1/2) {\color{black!50!white} $3+i$};

\node[font=\small, scale=0.7] at (\a1-7/2,\a2+3/2) {\color{black!50!white} $-3+2i$};
\node[font=\small, scale=0.7] at (\a1-5/2,\a2+3/2) {\color{black!50!white} $-2+2i$};
\node[font=\small, scale=0.7] at (\a1-3/2,\a2+3/2) {\color{black!50!white} $-1+2i$};
\node[font=\small, scale=0.7] at (\a1-1/2,\a2+3/2) {\color{black!50!white} $2i$};
\node[font=\small, scale=0.7] at (\a1+1/2,\a2+3/2) {\color{black!50!white} $1+2i$};
\node[font=\small, scale=0.7] at (\a1+3/2,\a2+3/2) {\color{black!50!white} $2+2i$};
\node[font=\small, scale=0.7] at (\a1+5/2,\a2+3/2) {\color{black!50!white} $3+2i$};

\node[font=\small, scale=0.7] at (\a1-7/2,\a2+5/2) {\color{black!50!white} $-3+3i$};
\node[font=\small, scale=0.7] at (\a1-5/2,\a2+5/2) {\color{black!50!white} $-2+3i$};
\node[font=\small, scale=0.7] at (\a1-3/2,\a2+5/2) {\color{black!50!white} $-1+3i$};
\node[font=\small, scale=0.7] at (\a1-1/2,\a2+5/2) {\color{black!50!white} $3i$};
\node[font=\small, scale=0.7] at (\a1+1/2,\a2+5/2) {\color{black!50!white} $1+3i$};
\node[font=\small, scale=0.7] at (\a1+3/2,\a2+5/2) {\color{black!50!white} $2+3i$};
\node[font=\small, scale=0.7] at (\a1+5/2,\a2+5/2) {\color{black!50!white} $3+3i$};

\node[font=\small, scale=0.7] at (\a1-7/2,\a2-3/2) {\color{black!50!white} $-3-i$};
\node[font=\small, scale=0.7] at (\a1-5/2,\a2-3/2) {\color{black!50!white} $-2-i$};
\node[font=\small, scale=0.7] at (\a1-3/2,\a2-3/2) {\color{black!50!white} $-1-i$};
\node[font=\small, scale=0.7] at (\a1-1/2,\a2-3/2) {\color{black!50!white} $-i$};
\node[font=\small, scale=0.7] at (\a1+1/2,\a2-3/2) {\color{black!50!white} $1-i$};
\node[font=\small, scale=0.7] at (\a1+3/2,\a2-3/2) {\color{black!50!white} $2-i$};
\node[font=\small, scale=0.7] at (\a1+5/2,\a2-3/2) {\color{black!50!white} $3-i$};

\node[font=\small, scale=0.7] at (\a1-7/2,\a2-5/2) {\color{black!50!white} $-3-2i$};
\node[font=\small, scale=0.7] at (\a1-5/2,\a2-5/2) {\color{black!50!white} $-2-2i$};
\node[font=\small, scale=0.7] at (\a1-3/2,\a2-5/2) {\color{black!50!white} $-1-2i$};
\node[font=\small, scale=0.7] at (\a1-1/2,\a2-5/2) {\color{black!50!white} $-2i$};
\node[font=\small, scale=0.7] at (\a1+1/2,\a2-5/2) {\color{black!50!white} $1-2i$};
\node[font=\small, scale=0.7] at (\a1+3/2,\a2-5/2) {\color{black!50!white} $2-2i$};
\node[font=\small, scale=0.7] at (\a1+5/2,\a2-5/2) {\color{black!50!white} $3-2i$};

\node[font=\small, scale=0.7] at (\a1-7/2,\a2-7/2) {\color{black!50!white} $-3-3i$};
\node[font=\small, scale=0.7] at (\a1-5/2,\a2-7/2) {\color{black!50!white} $-2-3i$};
\node[font=\small, scale=0.7] at (\a1-3/2,\a2-7/2) {\color{black!50!white} $-1-3i$};
\node[font=\small, scale=0.7] at (\a1-1/2,\a2-7/2) {\color{black!50!white} $-3i$};
\node[font=\small, scale=0.7] at (\a1+1/2,\a2-7/2) {\color{black!50!white} $1-3i$};
\node[font=\small, scale=0.7] at (\a1+3/2,\a2-7/2) {\color{black!50!white} $2-3i$};
\node[font=\small, scale=0.7] at (\a1+5/2,\a2-7/2) {\color{black!50!white} $3-3i$};
\end{tikzpicture}
\caption{The set $U$ is the central square labeled 0. The colored curves indicate the images under the map $z \mapsto \frac1z$ of the part of the boundary $\partial U$ with the same color. $\frac1U$ is outside the ``cloverleaf''.\\[-3.3mm]}
\end{subfigure}
\caption{For each $b \in \mathbb Z[i]\setminus\{0,\pm1, \pm i\}$ the map $T$ first maps the cylinder set $\Delta(b)$ in (A) to the region in (B) with the label $b$ and then shifts this region back to $U$ by subtracting the digit $b$.}
\label{fig:hurmap}
\end{figure}
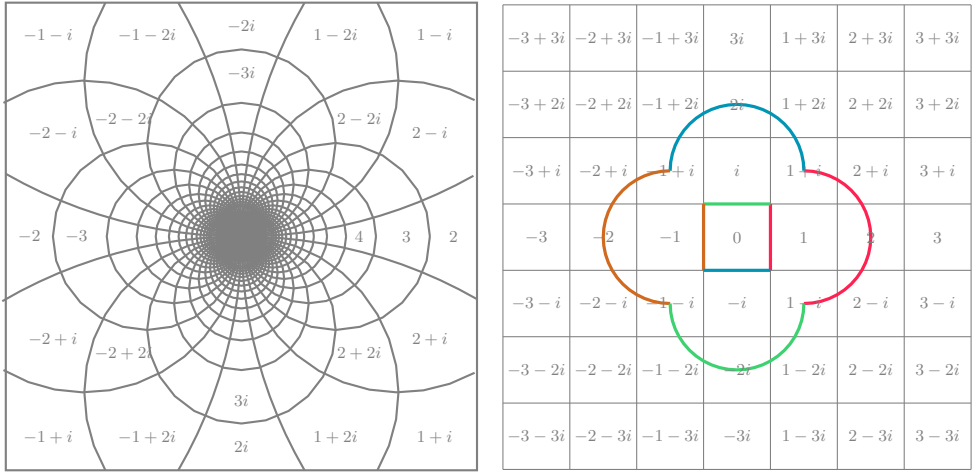
Figure~\ref{fig:hurmap} shows how $T$ acts on the sets $\Delta(b)$ for $b \in \mathbb Z[i] \setminus \{0,\pm1, \pm i\}$. 

For the Hurwitz algorithm not all infinite strings $a_1a_2\cdots$ of Gaussian integers $\mathbb Z[i]\setminus\{0,\pm1, \pm i\}$ occur as Hurwitz expansions. For example, as can be seen from Figure~\ref{fig:hurmap}, the pattern $a_na_{n+1}=2(-3)$ does not occur. As in the case of Zeckendorf expansions and the nearest integer algorithm, we need further restrictions to characterize the digit strings of Hurwitz continued fraction expansions. These restrictions are again of a simple type. Indeed, the collection 
\[
\{T^n(\Delta(b_1,\ldots,b_n)) \;:\; b_1,\ldots, b_n \in \mathbb Z[i],\; n\ge 2\}
\]
contains only finitely many sets up to Lebesgue null sets, in other words, the Hurwitz algorithm satisfies the finite range property. According to \cite[Lemma~3.6]{LV24} the finite range property is equivalent to the fact that the Hurwitz map admits a {\it finite partition} in the following sense. A collection $\mathcal P$ of Borel measurable subsets of $U$ is a {\em partition} of $U$ if the sets from $\mathcal P$ are pairwise disjoint and $U = \bigcup_{P \in \mathcal P} P$ up to sets of zero Lebesgue measure. A finite partition $\mathcal P$ of $U$ is called a {\it finite partition for the map $T$} if for any $b \in \mathbb Z[i] \setminus \{0,\pm1, \pm i\}$ and any $P \in \mathcal P$ the set $T(\Delta(b)\cap P)$ is a union of elements from $\mathcal P$, again up to sets of zero Lebesgue measure. The finite partition for the Hurwitz map $T$ is shown in Figure~\ref{fig:markov1/2,1/2} (see \cite[Section~2]{Ei2018})\footnote{By carefully inspecting the mapping properties of $T$ illustrated in Figure~\ref{fig:hurmap}, one easily sees that the partition in Figure~\ref{fig:markov1/2,1/2} is indeed a finite partition for $T$.}. 

In 1979 Jeffrey Shallit \cite[Section II.3]{Sha79} studied a variant of the Hurwitz algorithm by using a different ``floor function''. In particular, in \cite[Theorem II.3.3]{Sha79} the possible digit strings of this algorithm are characterized by a finite automaton.

\section{A Generalization: The $\balpha$-Hurwitz algorithm}\label{sec:alpha}

We are now ready to introduce the main object of our paper. We consider perturbations of the Hurwitz map. To be precise, we make the mapping $T$ defined in \eqref{eq:mapT} parameter dependent by shifting the domain $U$. For ${\balpha}=(\alpha_1,\alpha_2) \in \mathbb R^2$  set
\[
U_{\balpha} = \big\{z \in \C \;:\;  \alpha_1-1 \le \Re z < \alpha_1,\, \alpha_2-1 \le  \Im z < \alpha_2  \big\}
\]
and define the function 
$
\lfloor \cdot \rfloor_{\balpha} : \C \to U_{\balpha};\,  z  \mapsto z-w$, where $w\in\Z[i]$ is the unique Gaussian integer satisfying $z-w\in U_{\balpha}$. The {\it $\balpha$-Hurwitz map} is then given by
\[ 
T_{\balpha} : U_{\balpha}\setminus\{0\} \to U_{\balpha}; \quad z \mapsto \frac1z - \Big\lfloor \frac1z \Big\rfloor_{\balpha}.
\]
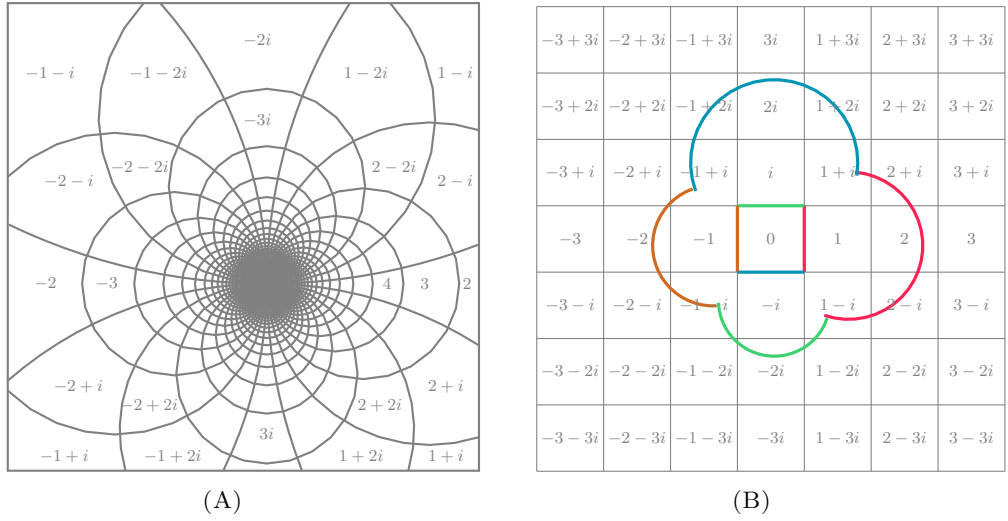
\begin{figure}[h]
\centering

\begin{subfigure}[b]{0.45\textwidth}
         \centering
\begin{tikzpicture}[scale=6.2]
\tikzmath{\a1 = 0.45; \a2 = 0.6;}

\foreach \y in {2,3,...,15} {
\draw [smooth, black!50,thick, domain=0:360] plot ({1/(2*(\a1+\y))+1/(2*(\a1+\y) )*cos(\x)}, {1/(2*(\a1+\y) )*sin(\x)});
\draw [smooth, black!50,thick, domain=0:360] plot ({1/(2*(\a2 +\y))*cos(\x)}, {-1/(2*(\a2+\y))+1/(2*(\a2+\y) )*sin(\x)});
}
\draw [smooth, black!50,thick, domain=147.3:201] plot ({1/(2*\a1)+1/(2*\a1)*cos(\x)}, {1/(2*\a1) )*sin(\x)});
\draw [smooth, black!50,thick, domain=72:288] plot ({1/(2*(\a1+1))+1/(2*(\a1+1) )*cos(\x)}, {1/(2*(\a1+1) )*sin(\x)});

\draw [smooth, black!50,thick, domain=244:291] plot ({1/(2*(1-\a2 ))*cos(\x)}, {1/(2*(1-\a2))+1/(2*(1-\a2) )*sin(\x)});
\draw [smooth, black!50,thick, domain=137:403] plot ({1/(2*(2-\a2 ))*cos(\x)}, {1/(2*(2-\a2))+1/(2*(2-\a2) )*sin(\x)});
\foreach \y in {3,4,...,16} {
\draw [smooth, black!50,thick, domain=0:360] plot ({1/(2*(\y-\a2 ))*cos(\x)}, {1/(2*(\y-\a2))+1/(2*(\y-\a2) )*sin(\x)});
\draw [smooth, black!50,thick, domain=0:360] plot ({1/(2*(\a1-\y))+1/(2*(\y-\a1) )*cos(\x)}, {1/(2*(\y-\a1) )*sin(\x)});
}

\draw [smooth, black!50,thick, domain=334:401.4] plot ({1/(2*(\a1-1))+1/(2*(1-\a1) )*cos(\x)}, {1/(2*(1-\a1) )*sin(\x)});
\draw [smooth, black!50,thick, domain=225:495] plot ({1/(2*(\a1-2))+1/(2*(2-\a1) )*cos(\x)}, {1/(2*(2-\a1) )*sin(\x)});

\draw [smooth, black!50,thick, domain=57.4:131.4] plot ({1/(2*\a2 )*cos(\x)}, {-1/(2*\a2)+1/(2*\a2 )*sin(\x)});
\draw [smooth, black!50,thick, domain=556:344] plot ({1/(2*(\a2 +1))*cos(\x)}, {-1/(2*(\a2+1))+1/(2*(\a2+1) )*sin(\x)});

\filldraw[fill=black!50, draw=black!50] (0, 0) circle (1.5pt);

\draw[thick,black!50!white] (\a1-1, \a2-1) -- (\a1-1, \a2) -- (\a1, \a2) -- (\a1, \a2 -1) -- cycle;

\node[font=\small, scale=0.7] at (\a1-.025,0) {\color{black!50} 2};
\node[font=\small, scale=0.7] at (\a1-.115,0) { \color{black!50} 3};
\node[font=\small, scale=0.7] at (\a1-1+.21,0) { \color{black!50} $-3$};
\node[font=\small, scale=0.7] at (\a1-.195,0) { \color{black!50} 4};
\node[font=\small, scale=0.7] at (\a1-1+.08,0) { \color{black!50} $-2$};
\node[font=\small, scale=0.7] at (-.02,\a2-.08) { \color{black!50} $-2i$};
\node[font=\small, scale=0.7] at (-.02,\a2-.25) { \color{black!50} $-3i$};
\node[font=\small, scale=0.7] at (0,\a2-1+.08) { \color{black!50} $3i$};
\node[font=\small, scale=0.7] at (.4,.45) { \color{black!50} $1-i$};
\node[font=\small, scale=0.7] at (.4,.22) { \color{black!50} $2-i$};
\node[font=\small, scale=0.7] at (.21,.45) { \color{black!50} $1-2i$};
\node[font=\small, scale=0.7] at (.27,.25) { \color{black!50} $2-2i$};

\node[font=\small, scale=0.7] at (.38,-.37) { \color{black!50} $1+i$};
\node[font=\small, scale=0.7] at (.38,-.22) { \color{black!50} $2+i$};
\node[font=\small, scale=0.7] at (.2,-.37) { \color{black!50} $1+2i$};
\node[font=\small, scale=0.7] at (.24,-.26) { \color{black!50} $2+2i$};

\node[font=\small, scale=0.7] at (-.43,-.37) { \color{black!50} $-1+i$};
\node[font=\small, scale=0.7] at (-.4,-.22) { \color{black!50} $-2+i$};
\node[font=\small, scale=0.7] at (-.2,-.37) { \color{black!50} $-1+2i$};
\node[font=\small, scale=0.7] at (-.25,-.26) { \color{black!50} $-2+2i$};

\node[font=\small, scale=0.7] at (-.46,.45) { \color{black!50} $-1-i$};
\node[font=\small, scale=0.7] at (-.42,.22) { \color{black!50} $-2-i$};
\node[font=\small, scale=0.7] at (-.23,.45) { \color{black!50} $-1-2i$};
\node[font=\small, scale=0.7] at (-.27,.25) { \color{black!50} $-2-2i$};
\end{tikzpicture}
\caption{}
\end{subfigure}
\hfill
\begin{subfigure}[b]{0.45\textwidth}
\centering
\begin{tikzpicture}[scale=.88]
\tikzmath{\a1 = 0.45; \a2 = 0.6;}

\draw[smooth, black!50!white] (\a1-4, \a2-4) -- (\a1-4, \a2+3);
\draw[smooth, black!50!white] (\a1-3, \a2-4) -- (\a1-3, \a2+3);
\draw[smooth, black!50!white] (\a1-2, \a2-4) -- (\a1-2, \a2+3);
\draw[smooth, black!50!white] (\a1-1, \a2-4) -- (\a1-1, \a2+3);
\draw[smooth, black!50!white] (\a1, \a2-4) -- (\a1, \a2+3);
\draw[smooth, black!50!white] (\a1+1, \a2-4) -- (\a1+1, \a2+3);
\draw[smooth, black!50!white] (\a1+2, \a2-4) -- (\a1+2, \a2+3);
\draw[smooth, black!50!white] (\a1+3, \a2-4) -- (\a1+3, \a2+3);

\draw[smooth, black!50!white] (\a1-4, \a2-4) -- (\a1+3, \a2-4);
\draw[smooth, black!50!white] (\a1-4, \a2-3) -- (\a1+3, \a2-3);
\draw[smooth, black!50!white] (\a1-4, \a2-2) -- (\a1+3, \a2-2);
\draw[smooth, black!50!white] (\a1-4, \a2-1) -- (\a1+3, \a2-1);
\draw[smooth, black!50!white] (\a1-4, \a2) -- (\a1+3, \a2);
\draw[smooth, black!50!white] (\a1-4, \a2+1) -- (\a1+3, \a2+1);
\draw[smooth, black!50!white] (\a1-4, \a2+2) -- (\a1+3, \a2+2);
\draw[smooth, black!50!white] (\a1-4, \a2+3) -- (\a1+3, \a2+3);

\draw [awesome,very thick, domain=251:445] plot ({1/(2*\a1)+1/(2*\a1 )*cos(\x)}, {1/(2*\a1 )*sin(\x)});
\draw[awesome,very thick] (\a1,\a2-1)--(\a1,\a2);
\draw [bondiblue,very thick, domain=-9:200] plot ({1/(2*(1-\a2 ))*cos(\x)}, {1/(2*(1-\a2))+1/(2*(1-\a2) )*sin(\x)});
\draw[ufogreen, very thick] (\a1, \a2) -- (\a1-1,\a2);
\draw [cocoabrown,very thick, domain=110:273] plot ({1/(2*(\a1-1))+1/(2*(1-\a1) )*cos(\x)}, {1/(2*(1-\a1) )*sin(\x)});
\draw [cocoabrown,very thick] (\a1-1,\a2)--(\a1-1,\a2-1);
\draw [ufogreen,very thick, domain=182:342] plot ({1/(2*\a2 )*cos(\x)}, {-1/(2*\a2)+1/(2*\a2 )*sin(\x)});
\draw[bondiblue,very thick] (\a1-1,\a2-1)--(\a1,\a2-1);

\node[font=\small, scale=0.7] at (\a1-7/2,\a2-1/2) {\color{black!50!white} $-3$};
\node[font=\small, scale=0.7] at (\a1-5/2,\a2-1/2) {\color{black!50!white} $-2$};
\node[font=\small, scale=0.7] at (\a1-3/2,\a2-1/2) {\color{black!50!white} $-1$};
\node[font=\small, scale=0.7] at (\a1-1/2,\a2-1/2) {\color{black!50!white} $0$};
\node[font=\small, scale=0.7] at (\a1+1/2,\a2-1/2) {\color{black!50!white} $1$};
\node[font=\small, scale=0.7] at (\a1+3/2,\a2-1/2) {\color{black!50!white} $2$};
\node[font=\small, scale=0.7] at (\a1+5/2,\a2-1/2) {\color{black!50!white} $3$};

\node[font=\small, scale=0.7] at (\a1-7/2,\a2+1/2) {\color{black!50!white} $-3+i$};
\node[font=\small, scale=0.7] at (\a1-5/2,\a2+1/2) {\color{black!50!white} $-2+i$};
\node[font=\small, scale=0.7] at (\a1-3/2,\a2+1/2) {\color{black!50!white} $-1+i$};
\node[font=\small, scale=0.7] at (\a1-1/2,\a2+1/2) {\color{black!50!white} $i$};
\node[font=\small, scale=0.7] at (\a1+1/2,\a2+1/2) {\color{black!50!white} $1+i$};
\node[font=\small, scale=0.7] at (\a1+3/2,\a2+1/2) {\color{black!50!white} $2+i$};
\node[font=\small, scale=0.7] at (\a1+5/2,\a2+1/2) {\color{black!50!white} $3+i$};

\node[font=\small, scale=0.7] at (\a1-7/2,\a2+3/2) {\color{black!50!white} $-3+2i$};
\node[font=\small, scale=0.7] at (\a1-5/2,\a2+3/2) {\color{black!50!white} $-2+2i$};
\node[font=\small, scale=0.7] at (\a1-3/2,\a2+3/2) {\color{black!50!white} $-1+2i$};
\node[font=\small, scale=0.7] at (\a1-1/2,\a2+3/2) {\color{black!50!white} $2i$};
\node[font=\small, scale=0.7] at (\a1+1/2,\a2+3/2) {\color{black!50!white} $1+2i$};
\node[font=\small, scale=0.7] at (\a1+3/2,\a2+3/2) {\color{black!50!white} $2+2i$};
\node[font=\small, scale=0.7] at (\a1+5/2,\a2+3/2) {\color{black!50!white} $3+2i$};

\node[font=\small, scale=0.7] at (\a1-7/2,\a2+5/2) {\color{black!50!white} $-3+3i$};
\node[font=\small, scale=0.7] at (\a1-5/2,\a2+5/2) {\color{black!50!white} $-2+3i$};
\node[font=\small, scale=0.7] at (\a1-3/2,\a2+5/2) {\color{black!50!white} $-1+3i$};
\node[font=\small, scale=0.7] at (\a1-1/2,\a2+5/2) {\color{black!50!white} $3i$};
\node[font=\small, scale=0.7] at (\a1+1/2,\a2+5/2) {\color{black!50!white} $1+3i$};
\node[font=\small, scale=0.7] at (\a1+3/2,\a2+5/2) {\color{black!50!white} $2+3i$};
\node[font=\small, scale=0.7] at (\a1+5/2,\a2+5/2) {\color{black!50!white} $3+3i$};

\node[font=\small, scale=0.7] at (\a1-7/2,\a2-3/2) {\color{black!50!white} $-3-i$};
\node[font=\small, scale=0.7] at (\a1-5/2,\a2-3/2) {\color{black!50!white} $-2-i$};
\node[font=\small, scale=0.7] at (\a1-3/2,\a2-3/2) {\color{black!50!white} $-1-i$};
\node[font=\small, scale=0.7] at (\a1-1/2,\a2-3/2) {\color{black!50!white} $-i$};
\node[font=\small, scale=0.7] at (\a1+1/2,\a2-3/2) {\color{black!50!white} $1-i$};
\node[font=\small, scale=0.7] at (\a1+3/2,\a2-3/2) {\color{black!50!white} $2-i$};
\node[font=\small, scale=0.7] at (\a1+5/2,\a2-3/2) {\color{black!50!white} $3-i$};

\node[font=\small, scale=0.7] at (\a1-7/2,\a2-5/2) {\color{black!50!white} $-3-2i$};
\node[font=\small, scale=0.7] at (\a1-5/2,\a2-5/2) {\color{black!50!white} $-2-2i$};
\node[font=\small, scale=0.7] at (\a1-3/2,\a2-5/2) {\color{black!50!white} $-1-2i$};
\node[font=\small, scale=0.7] at (\a1-1/2,\a2-5/2) {\color{black!50!white} $-2i$};
\node[font=\small, scale=0.7] at (\a1+1/2,\a2-5/2) {\color{black!50!white} $1-2i$};
\node[font=\small, scale=0.7] at (\a1+3/2,\a2-5/2) {\color{black!50!white} $2-2i$};
\node[font=\small, scale=0.7] at (\a1+5/2,\a2-5/2) {\color{black!50!white} $3-2i$};

\node[font=\small, scale=0.7] at (\a1-7/2,\a2-7/2) {\color{black!50!white} $-3-3i$};
\node[font=\small, scale=0.7] at (\a1-5/2,\a2-7/2) {\color{black!50!white} $-2-3i$};
\node[font=\small, scale=0.7] at (\a1-3/2,\a2-7/2) {\color{black!50!white} $-1-3i$};
\node[font=\small, scale=0.7] at (\a1-1/2,\a2-7/2) {\color{black!50!white} $-3i$};
\node[font=\small, scale=0.7] at (\a1+1/2,\a2-7/2) {\color{black!50!white} $1-3i$};
\node[font=\small, scale=0.7] at (\a1+3/2,\a2-7/2) {\color{black!50!white} $2-3i$};
\node[font=\small, scale=0.7] at (\a1+5/2,\a2-7/2) {\color{black!50!white} $3-3i$};
\end{tikzpicture}
\caption{}
\end{subfigure}
\caption{The analog of Figure~\ref{fig:hurmap} for $T_\balpha$ with $\balpha=(\tfrac{9}{20}, \tfrac{3}{5})$: The cloverleaf turns into an asymmetric butterfly.} \label{fig:smallalpha}
\end{figure}

\noindent Let $z \in U_{\balpha}$. For each $n \in\mathbb{N}$, such that the iterates $T_{\balpha}^{k-1}(z)$ do not vanish for $k\in\{1,\ldots,n\}$, we set $a_k = \Big\lfloor \frac1{T_{\balpha}^{k-1}(z)} \Big\rfloor_{\balpha} \in \mathbb Z[i]$ and gain 
\begin{equation}\label{q:finitefrac}
z = \cfrac1{a_1 + \cfrac1{a_2 + \cfrac1{\ddots + \cfrac1{a_n + T^n_{\balpha}(z)}}}}.
\end{equation}
If $T_{\balpha}^n(z)$ vanishes, then \eqref{q:finitefrac} becomes a finite {\it $\balpha$-Hurwitz expansion} of $z$. If $T_{\balpha}^n(z)$ does not vanish for any $n\ge 0$ we can consider the limit $[a_1,a_2,\ldots]_{\balpha}$ of the right hand side of \eqref{q:finitefrac} for $n\to\infty$.  This is the {\it $\balpha$-Hurwitz algorithm} (which agrees with the Hurwitz algorithm from Section~\ref{sec:hurwitz} for $\balpha=(\frac12,\frac12)$). However, we have to be careful. {\it A priori} it is not clear if $[a_1,a_2,\ldots]_{\balpha}$ represents $z$ or if this limit exists at all. Indeed, $z=[a_1,a_2,\ldots]_{\balpha}$ only holds for particular parameters $\balpha$. To characterize this set of parameters we need some notation. Let $B_{\varepsilon}(a,b) \subseteq \R^2$ be the open ball with center $(a,b)$ and radius $\varepsilon$, and for a set $B \subseteq \R^2$ use $\overline B$ to denote its closure. Define the set
\[
\mathcal D = B_1(0,0) \cap \overline{B_1(1,0)} \cap \overline{B_1(0,1)} \cap \overline{B_1(1,1)},
\]
which is illustrated in Figure~\ref{fig:admissible}. We have the following result.

\begin{lemma}\label{lem:D}
Let $\balpha \in \mathbb{R}^2$ be given. Then the following assertions are equivalent.
\begin{itemize}
\item $\balpha\in\mathcal{D}$.
\item For each $z\in U_{\balpha}$ satisfying $T_\balpha^n(z)\neq 0$ for each $n\in \mathbb{N}$, the $\balpha$-Hurwitz algorithm is convergent and $z = [a_1, a_2, \ldots]_{\balpha}$.
\end{itemize}
\end{lemma}

\begin{proof}
This is an immediate consequence of \cite[Theorem 2.2]{dani2014continued} (see also \cite{cijsouw2015complex} and \cite[Section 5.2]{mastersthesis}). 
\end{proof}

Because of Lemma~\ref{lem:D} we will confine ourselves to parameters $\balpha\in\mathcal{D}$. 
\begin{figure}[h]
         \centering
         \includegraphics[width=.3\textwidth]{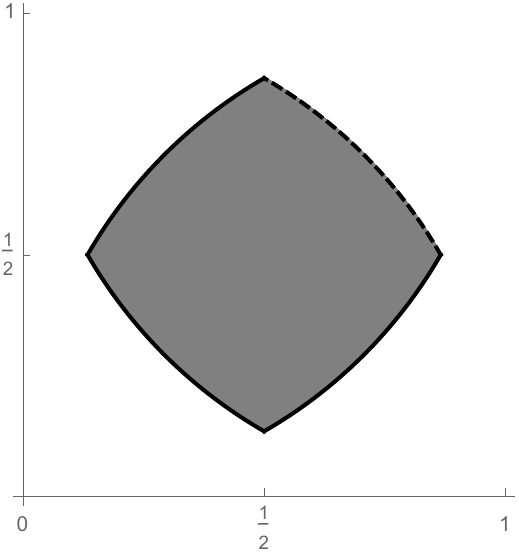}
 	\caption{The region $\mathcal D$ of parameters $\balpha$ for which $T_{\balpha}$ yields continued fraction expansions.}
  \label{fig:admissible}
\end{figure}

Fix $\balpha\in\mathcal{D}$. As we did for the Hurwitz map, we define cylinder sets for the $\balpha$-Hurwitz map $T_{\balpha}$. For each $b \in \mathbb Z[i]$ set 
$
\Delta_{\balpha}(b) = \big\{ z \in U_{\balpha} \, : \, \left\lfloor \frac1z \right\rfloor_{\balpha} = b\big\}
$
and let $T_{\balpha,b}$ be the restriction of $T_{\balpha}$ to $\Delta_{\balpha}(b)$. Then, for $b_1,\ldots, b_n \in \mathbb Z[i]$, $n\ge 2$, we define 
\[
\Delta_{\balpha}(b_1,\ldots, b_n) = T_{\balpha,b_1}^{-1}(\Delta_{\balpha}(b_2,\ldots, b_n)) \cap U_{\balpha}.
\]
Following \cite{Yur95} we say that the $\balpha$-Hurwitz algorithm satisfies the {\em finite range condition} if the collection
\[
\{T_{\balpha}^n(\Delta_{\balpha}(b_1,\ldots,b_n)) \;:\; b_1,\ldots, b_n \in \mathbb Z[i],\; n\ge 2\}
\]
contains only finitely many sets up to Lebesgue null sets. 
As mentioned before, by \cite[Lemma~3.6]{LV24} also for these ${\balpha}$-Hurwitz algorithms the finite range condition is equivalent to the fact that there exists a finite partition $\mathcal P$ of $U_{\balpha}$ for the map $T_{\balpha}$ in the following sense. For any $b \in \mathbb Z[i]$ for which $\Delta_{\balpha}(b) \neq \emptyset$ and any $P \in \mathcal P$ the set $T_{\balpha}(\Delta_{\balpha}(b)\cap P)$ is a union of elements from $\mathcal P$ up to sets of zero Lebesgue measure. In~\cite{LV23,LV24} the authors asked whether it is possible to identify sets of parameters $\balpha$ for which $T_{\balpha}$ admits a finite partition\footnote{In \cite{LV24} a map $T_{\balpha}$ with this property is called {\em serendipitous}.}. Numerical computations suggested that $T_{\balpha}$ admits a finite partition at least for each $\balpha\in\mathbb{Q}^2$. Figures~\ref{fig:markov1/2,1/2} and~\ref{fig:markov} provide some examples. In~\cite{mastersthesis}, a finite partition for $\balpha =\left(\frac{1}{3},\frac{1}{2}\right)$ is constructed in full detail. In this article we establish the following general result.

\begin{theorem}\label{t:mp}
Let $p,q,r,s \in \mathbb N$ be such that $(\frac{p}{q}, \frac{r}{s}) \in \mathcal D$. For $\balpha = (\frac{p}{q}, \frac{r}{s})$ the $\balpha$-Hurwitz algorithm satisfies the finite range condition. Equivalently, the map $T_{\balpha}$ admits a finite partition of $U_{\balpha}$.
\end{theorem}


\begin{figure}[h]
     \centering
     \begin{subfigure}[b]{0.3\textwidth}
         \centering
         \includegraphics[width=\textwidth]{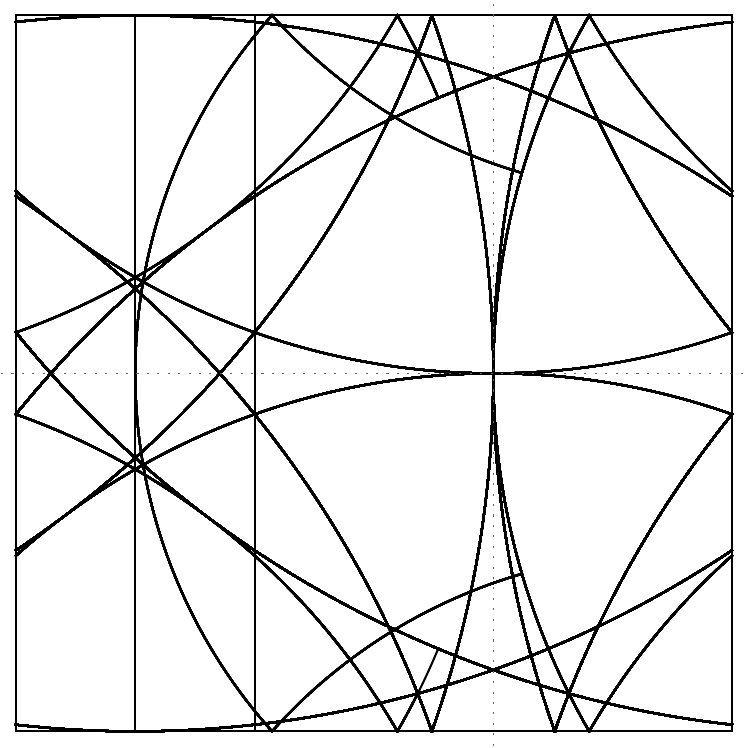}
         \caption{$\balpha=(\frac23,\frac12)$.} \label{fig:markov2/3,1/2}
     \end{subfigure}
     \hfill
     \begin{subfigure}[b]{0.3\textwidth}
         \centering
         \includegraphics[width=\textwidth]{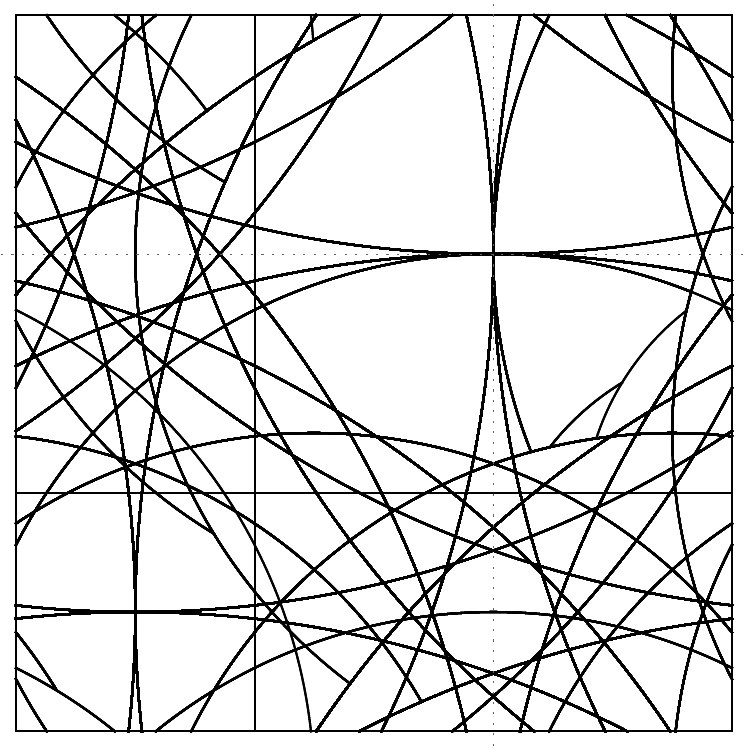}
         \caption{$\balpha=(\frac23,\frac23)$.} \label{fig:markov2/3,2/3}
     \end{subfigure}
     \hfill
     \begin{subfigure}[b]{0.3\textwidth}
         \centering
         \includegraphics[width=\textwidth]{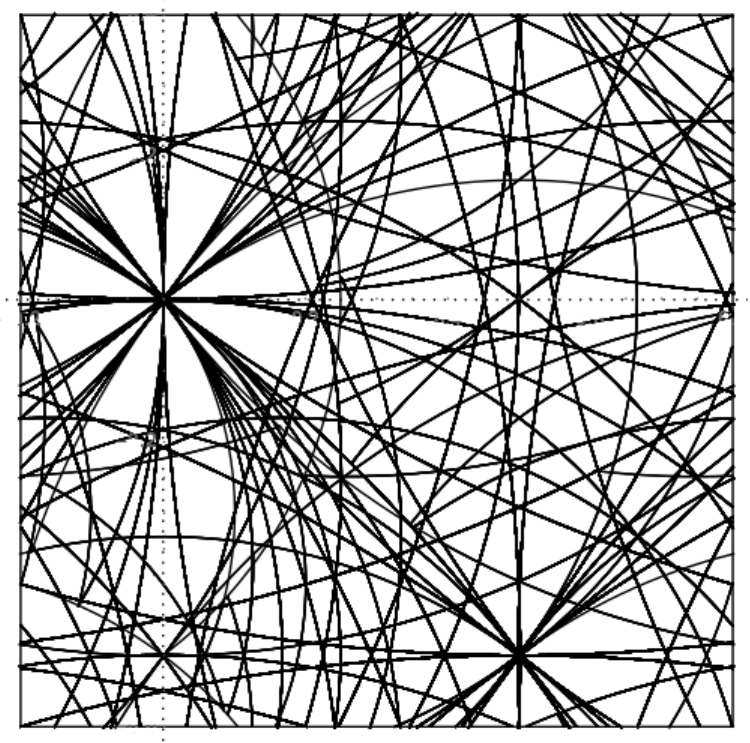}
         \caption{$\balpha=(\frac15,\frac35)$.} \label{fig:markov1/5,3/5}
     \end{subfigure}
 	\caption{The finite partition for $T_\balpha$ for various cases of $\balpha$.}
        \label{fig:markov}
\end{figure}

The proof of Theorem~\ref{t:mp} is given in the next section.

\section{Proof of the main result} \label{sec:markov}

Before we start the proof we recall some mapping properties of circles and lines in $\mathbb{C}$. For $r \in (0,\infty)$ and $m \in \C$ write $C_r(m) \subseteq \C$ for the circle with radius $r$ and center $m$. A subset of $\mathbb C$ that is either a circle or a line will be called a {\it generalized circle} and can be written in the form
\begin{equation}\label{eq:complexCircle}
M(a,b,c)= \big\{ z\in \mathbb{C} \;:\; a z\overline{z} - \overline{b}z - b\overline{z} + c = 0\big\},
\end{equation}
with $a,c\in\R$ and $b\in \C$ satisfying $|b|^2 > ac$ (see {\it e.g.}~\cite[Section~3.2]{Hahn:94}). In particular, if $a=0$ then $M(0,b,c)$ is a line with (possibly infinite) slope $-\frac{\Re b}{\Im b}$, while for $a \neq 0$ the set $M(a,b,c)$ is a circle with radius $r=\frac{\sqrt{|b|^2 - ac}}{|a|}$ and center $m=\frac{b}{a}$, {\it i.e.},
\begin{equation}\label{eq:transform}
M(a,b,c)=C_{\frac{\sqrt{|b|^2 - ac}}{|a|}}\Big(\frac{b}{a}\Big).
\end{equation}
The representation \eqref{eq:complexCircle} is particularly convenient when taking reciprocals of generalized circles. Indeed, direct calculation shows that
\begin{equation}\label{eq:1over}
\frac{1}{M(a,b,c)} = M(c,\overline{b},a).
\end{equation}
Also the effect of translation by a complex number can be seen in terms of the representation \eqref{eq:complexCircle}. Indeed, again by direct calculation, for each $z\in\C$ we gain the identity 
\begin{equation}\label{eq:translatey}
M(a,b,c) - z = M(a, b-az, a z\overline{z} - \overline bz - b\overline z + c).
\end{equation}

We can now start with the proof of Theorem~\ref{t:mp}. Let $p,q,r,s\in\mathbb{N}$ be given in a way that $\balpha=(\frac{p}{q}, \frac{r}{s}) \in \mathcal D \cap \mathbb Q^2$. In view of \cite[Lemma~3.6]{LV24}, it suffices to prove that $T_{\balpha}$ admits a finite partition. For convenience, we will enlarge the domain of $T_\balpha$ to $\overline{U}_\balpha\setminus\{0\}$ by using the same formula $z \mapsto \frac1z - \lfloor \frac1z \rfloor_{\balpha}$ on the whole boundary of the half-open square $U_\balpha$. The boundary of $U_{\balpha}$ consists of four line segments each of which is contained in one of the four lines of the collection
\begin{equation*}\label{q:boundaryU2}
 \mathcal{G}_0= 
\big\{ M(0,-si,2(r-s)),\; M(0,-si,2r),\;  M(0,q,2(p-q)), \;  M(0,q,2p) \big\}.  
\end{equation*}

We need the following criterion for the existence of a finite partition for $T_\balpha$. 
\begin{lemma}\label{lem:Tfin}
If 
\begin{equation} \label{eq:invariance}
S = \bigcup_{j=0}^\infty T_{\balpha}^j (\partial U_{\balpha} )  = \bigcup_{j=0}^\infty\bigcup_{G_0\in \mathcal{G}_0} T_{\balpha}^j (G_0)
\end{equation}
is contained in a finite union of generalized circles then the transformation $T_{\balpha}$ admits a finite partition.
\end{lemma}

\begin{proof}
Define $\mathcal{P}_n$, $n\ge 0$, recursively by $\mathcal{P}_0= \{T_\balpha(\Delta_{\balpha}(b)) \; :\; b\in \mathbb{Z}[i]\}$ and
\[
 \mathcal{P}_n = \big\{T_\balpha(\Delta_{\balpha}(b) \cap P) \; :\; b\in \mathbb{Z}[i],\, P\in \mathcal{P}_{n-1} \big\} \qquad (n\ge 1).
\]
For each $b \in \mathbb Z[i]$ we have $\partial T_{\balpha}(\Delta_{\balpha}(b)) \subset \partial U_{\balpha} \cup T_{\balpha}(\partial U_{\balpha})$ (see Figure~\ref{fig:smallalpha}). Thus,  $\mathcal{P}_0$ is a partition of $U_\balpha$ and $(\mathcal{P}_n)_{n\ge 0}$ is a nested sequence of partitions of $U_\balpha$. For each $P\in \mathcal{P}_n$ we have by induction that
\[
\partial P \subset  \bigcup_{j=0}^{n+1}T_{\balpha}^j (\partial U_{\balpha} ) \subset S. 
\]
By assumption, $S$ is a subset of a finite union of generalized circles. These circles and lines divide $U_\balpha$ into finitely many pieces. This implies that $\#\mathcal{P}_n$ is finite and uniformly bounded in $n$. Thus $(\mathcal{P}_n)_{n\ge 0}$ eventually stabilizes and, hence,  
\[
\mathcal{P}=\bigcup_{n\ge 0} \mathcal{P}_n
\]
is a finite partition of $U_\balpha$. Since by construction we have $T_\balpha(\Delta_{\balpha}(b) \cap P)$ is a union of elements of $\mathcal{P}$ for each $b\in\mathbb{Z}[i]$ and each $P\in \mathcal{P}$, the result follows.
\end{proof}

Let $S$ be as in \eqref{eq:invariance}. By the definition of $T_\balpha$, we have 
\begin{equation}\label{eq:SU}
S\subset \bigcup_{G\in \mathcal{U}}G, 
\end{equation}
where $\mathcal{U}$ is the collection of all generalized circles $G$ that satisfy the following condition: There are $N\ge 1$, $G_0\in \mathcal{G}_0$, and $z_0,\ldots, z_{N-1}\in \mathbb {Z}[i]$ such that the generalized circles $G_1,\ldots,G_N$ defined recursively by $G_{n} = \frac1{G_{n-1}} - z_{n-1}$ satisfy $G_n\cap U_\balpha\not=\emptyset$ for $n\in\{1,\ldots,N\}$ and $G_{N-1}=G$.\footnote{Note that, although already $G_{N-1}=G$ we need to go one step further for technical reasons. This does not restrict the choices of $G$ because it is always possible to choose $z_{N-1}$ in a way that $G_{N}\cap U_\balpha\not=\emptyset$.} In view of Lemma~\ref{lem:Tfin} it suffices to show that $\mathcal{U}$ is a finite collection.

Let $G\in \mathcal{U}$ be arbitrary and let $G_0,G_1,\ldots,G_N$ be as above with $G=G_{N-1}$. Since $G_0\in \mathcal{G}_0$, we have $G_0=M(2A_{-1}, B_0,2A_0)$ with $A_{-1}=0$, $A_0\in \mathbb Z\setminus \{0\}$, and $B_0 \in \mathbb Z[i]$. Recursively define the sequences $(A_n)_{-1 \le n\le N}$ and $(B_n)_{0\le n\le  N}$ by
\begin{equation}\label{q:rec}\begin{split}
2A_{n+1} = \ & 2A_n z_n \overline{z_n} - B_n z_n -  \overline{B_n} \overline{z_n} + 2A_{n-1},\\
B_{n+1} =\ & \overline{B_n} - 2 A_nz_n.
\end{split}\end{equation}

\begin{lemma}\label{l:recursive}
Let $G_0,\ldots,G_N$ be given as above and let $n\in \{0,\ldots, N\}$. We have $G_n = M(2A_{n-1},B_n,2A_n)$, where $A_n, B_n$ are defined by the recurrence \eqref{q:rec} and $A_{-1}=0$, $A_n \in \mathbb Z$, and $B_n \in \mathbb Z[i]$.
\end{lemma}

\begin{proof}
We prove the result inductively. We have $A_{-1}=0$ and for $n=0$ all claims of the lemma are true. Suppose that all claims are true for each $k\in\{0,\ldots,n\}$ with $n\in\{0,\ldots, N-1\}$. Then $G_n = M(2A_{n-1},B_n, 2A_{n})$. Therefore, by \eqref{eq:1over} we have $\tfrac1{G_n} = M(2A_n,\overline{B_n}, 2A_{n-1})$. Applying \eqref{eq:translatey} now implies, together with \eqref{q:rec}, that
\footnote{If $\frac1{G_n}$ is a line then there are infinitely many choices for $z_n$ satisfying $G_{n+1}=\frac1{G_n}-z_n$ for a fixed $G_{n+1}$. However, in this case $A_n=0$ and, because all the possible values of $z_n$ must lie on a line parallel to $\frac1{G_n}$, the quantity $B_nz_n + \overline{B_n}\overline{z_n}$ is the same for each of these choices. Thus the values $A_{n+1}$ and $B_{n+1}$ do not depend on the choice of $z_n$.}
\[
\begin{split}
 G_{n+1} = \frac{1}{G_n}-z_n &= M(2A_n,\overline{B_n}-2A_nz_n, 2A_nz_n\overline{z_n}-B_nz_n-\overline{B_n}\overline{z_n} + 2A_{n-1}) \\
 &= M(2A_{n},B_{n+1}, 2A_{n+1}). 
\end{split}
\]
Because $A_{-1}=0$, $A_k \in \mathbb Z$, and $B_k \in \mathbb Z[i]$ for all $k\in \{0,\ldots, n\}$ by assumption, we have $2A_{n+1}=2A_n |z_n|^2 - 2\Re(B_n z_n) + 2A_{n-1} \in 2\mathbb{Z}$, hence, $A_{n+1} \in \mathbb{Z}$, and $B_{n+1}=\overline{B_n} - 2 A_n\overline{z_n} \in \mathbb Z[i]$. This concludes the induction step.
\end{proof}

We claim that for each $n \in \{0,\ldots, N\}$ we have 
\begin{equation}\label{q:B0}
B_n \overline{B_n} - 4A_{n-1}A_n = |B_0|^2.
\end{equation}
Indeed, this follows by induction because \eqref{q:B0} trivially holds for $n=0$ and, using the recurrence relations from \eqref{q:rec}, we have
\[
\begin{split}
B_{n+1} \overline{B_{n+1}} - 4A_nA_{n+1} &= (\overline{B_n} - 2 A_nz_n)(B_n - 2 A_n \overline{z_n}) \\
&\qquad\qquad -2A_n(2A_n z_n \overline{z}_n - B_n z_n -  \overline{B_n}\overline{z_n} + 2A_{n-1}) \\
&\;=B_n \overline{B_n} - 4A_{n-1}A_n,
\end{split}
\]
for $n \in \{0,\ldots, N-1\}$.

If $G_n$ is a circle, we get the following formula of its radius $\rho_n$.

\begin{lemma}\label{l:radius}
Let $G_0,\ldots, G_N$ be as above. For $n\in\{0,\ldots, N\}$ with $G_n$ a circle 
the radius $\rho_n$ of this circle satisfies the identity
\begin{equation}\label{q:rhon}
\rho_n^2 = \frac{|B_0|^2}{4A_{n-1}^2}.
\end{equation}
\end{lemma}

\begin{proof}
Recall that $G_n=M(2A_{n-1},B_n,2A_n)$ is a circle if and only if $A_{n-1} \neq 0$. In that case by \eqref{eq:transform} and \eqref{q:B0} its radius $\rho_n$ satisfies
\begin{equation*}\label{q:rhon2}
\rho_n^2 = \frac{B_n \overline{B_n} - 4A_{n-1}A_n}{4A_{n-1}^2} = \frac{|B_0|^2}{4A_{n-1}^2}. \qedhere
\end{equation*}
\end{proof}

The next lemma bounds the radii $\rho_n$ that can occur in the sequence $(G_n)_{0\le n \le N}$.

\begin{lemma}\label{l:rho}
Let $\rho_{\min} = \tfrac{\min\{|x-y|\,:\, x\in U_{\balpha},\, y\in C_1(0)\}}2$ and let $G_0,\ldots, G_N$ be as above. For $n\in\{0,\ldots, N\}$ with $G_n$ a circle we have $\rho_{\min} < \rho_n \le \frac{|B_0|}{2}$ for its radius $\rho_n$.
\end{lemma}

\begin{figure}[h]
\begin{subfigure}[b]{0.47\textwidth}
\centering
\begin{tikzpicture}[scale=1.8]
\filldraw[thick, draw=black, fill=white] (-2/3,-2/5)--(1/3,-2/5)--(1/3,3/5)--(-2/3,3/5)--cycle;
\draw(-1.2,0)--(1.2,0)(0,-1.2)--(0,1.2);
\draw[thick] (0,0) circle (1cm);
\draw[thick, gray] (-2/3,3/5)--(-.7433,.6689);
\node[gray] at (-1,.7) {\small $2\rho_{\min}$};
\node at (-.2,.2) {\scriptsize $U_\alpha$};
\node at (.65,1) {\scriptsize $C_1(0)$};
\draw[thick, gray] (.5,.33) circle (.3cm);
\node[gray] at (.33,.73) {\scriptsize $G_n$};
\end{tikzpicture}
\caption{$G_n\cap C_1(0)= \emptyset$}
\end{subfigure}
\hfill
\begin{subfigure}[b]{0.47\textwidth}
\centering
\begin{tikzpicture}[scale=1.8]
\filldraw[thick, draw=black, fill=white] (-2/3,-2/5)--(1/3,-2/5)--(1/3,3/5)--(-2/3,3/5)--cycle;
\draw(-1.2,0)--(1.2,0)(0,-1.2)--(0,1.2);
\draw[thick] (0,0) circle (1cm);
\draw[thick, gray] (-2/3,3/5)--(-.7433,.6689);
\node[gray] at (-1,.7) {\small $2\rho_{\min}$};
\node at (.17,.2) {\scriptsize $U_\alpha$};
\node at (.65,1) {\scriptsize $C_1(0)$};
\draw[thick, gray] (.19,-1.2)--(-.55,1.2);
\node[gray] at (-.68,1.15) {\scriptsize $G_n$};
\filldraw[fill=gray, draw=gray] (-.3,.4) circle (.04cm);
\node[gray] at (-.35,.3) {\scriptsize $a$};
\filldraw[fill=gray, draw=gray] (-.45,.88) circle (.04cm);
\node[gray] at (-.35,.82) {\scriptsize $b$};
\end{tikzpicture}
\caption{$G_n\cap C_1(0) \neq \emptyset$}
\end{subfigure}
\caption{The two cases for the position of $G_n$ relative to $C_1(0)$ from the proof of Lemma~\ref{l:rho}}
\label{f:lemmarho}
\end{figure}
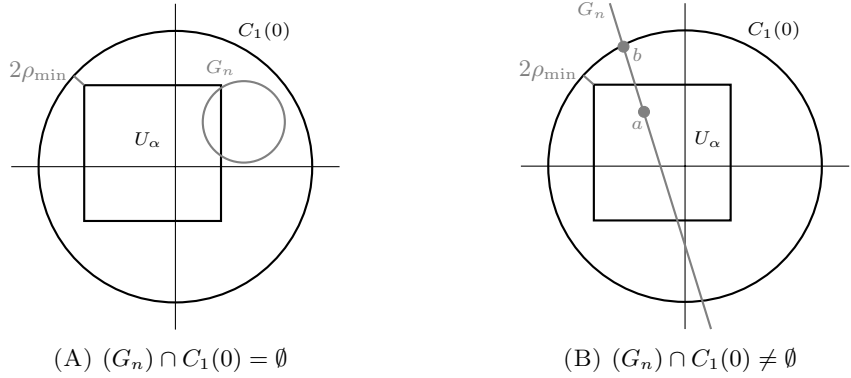

\begin{proof}
The upper bound follows from \eqref{q:rhon} because $A_{n-1} \in \mathbb Z \setminus \{0\}$. The lower bound will be proved by induction. First note that for $n=0$ there is nothing to prove because $G_0$ is not a circle. Let $n \in \{0,\ldots,N-1\}$ and assume that the lower bound holds for all $k\in\{0,\ldots n\}$ for which $G_k$ is a circle. We assume that $G_{n+1}$ is a circle, since otherwise there is nothing to prove. By definition, $G_n \cap \overline{U_\balpha}\not=\emptyset$. We distinguish two cases. If $G_n$ does not intersect the unit circle $C_1(0)$, $G_n$ is a circle inside $C_1(0)$, see Figure~\ref{f:lemmarho}~(\textsc{a}). Thus $G_{n+1}=\frac1{G_n}-z_n$ implies that $\rho_n < \rho_{n+1}$, because the function $z\mapsto \frac1z - z_n$ increases distances of arguments inside $C_1(0)$, and the result follows from the induction hypothesis. If $G_n$ intersects $C_1(0)$ there exist $a,b \in G_n$ such that $a \in U_{\balpha}$ and $b \in C_1(0)$, see Figure~\ref{f:lemmarho}~(\textsc{b}). Then $|a| < 1-2\rho_{\min}$, and, hence, we have
$
\big| \frac1{a} \big|> \frac1{1-2\rho_{\min}} > 1+2\rho_{\min}.
$
Since $| \frac1{b} | =1$, this yields $|(\frac1{a}-z_n) - (\frac1{b}-z_n)|
=|\frac1{a} - \frac1{b}| 
> 2\rho_{\min}$. But because $G_{n+1}=\frac1{G_n}-z_n$ we have $\frac1a-z_n, \frac1b-z_n \in G_{n+1}$. Thus  $\rho_{n+1} > \rho_{\min}$ holds also in this case and the lemma is proved. 
\end{proof}

We are now in a position to prove Theorem~\ref{t:mp}.
\begin{proof}[Proof of Theorem~\ref{t:mp}]
Let $G_0,\ldots,G_N$ with $G_{N-1}=G$ be as above. Lemma~\ref{l:rho} and \eqref{q:rhon} together imply that there is a constant $M_{\balpha}$ that depends only on $\balpha$ such that $|A_{n}| \in \{0,1, \ldots, M_{\balpha}\}$ for all $n \in\{-1,\ldots, N-1\}$ (note that $A_n=0$ if $G_{n+1}$ is a line). In view of \eqref{q:B0} this implies that $|B_n|^2 = |B_0|^2 + 4A_{n-1}A_n$ can only attain a finite number of values. Because $B_n \in \mathbb Z[i]$ this entails that there exists a finite set $F_{\balpha} \subset \mathbb Z[i]$ that only depends on $\balpha$ such that $B_n \in F_{\balpha}$ for all $n \in\{0,\ldots, N-1\}$. Thus, because $G=G_{N-1}$ we have in particular that $G\in \mathcal{U'}$, where  
\[
\begin{split}
\mathcal{U'} = \big\{M(2A_{N-1},B_{N-1},2A_{N-2})  \,:\, 
A_{N-1},A_{N-2} \in  \{0,1, \ldots, M_{\balpha}\},\, B_{N-1} \in F_\mathcal{\balpha} \big\}
\end{split}
\]
is a finite collection. Since $G$ was an arbitrary element of $\mathcal{U}$ this implies that $\mathcal{U} \subset \mathcal{U'}$. Thus $\mathcal{U}$ is finite and, hence, by \eqref{eq:SU}, $S$ is contained in a finite union of generalized circles. Thus the theorem follows from Lemma~\ref{lem:Tfin}.
\end{proof}

\section{Perspectives}\label{sec:perspectives}
As emphasized in \cite[Lemma~3.7 and Remark~3.8]{LV24}, in the one-dimensional case of {\it $\alpha$-continued fraction algorithms} (see \cite{nakada1981metrical}) the associated mapping $T_\alpha$ admits a finite partition if and only if $\alpha$ is either a rational or a quadratic irrational number. For the $\balpha$-Hurwitz algorithm \cite[Corollary~3.9]{LV24} gives a class of non-quadratic irrational parameters  for which this property does not hold. However, it remains open if Theorem~\ref{t:mp} can be extended to parameters $\balpha$ with quadratic irrational coordinates.

As mentioned before, Theorem~\ref{t:mp} paves the way for further explorations. For instance, Richard Lakein~\cite{lakein1973approximation} determines the optimal constant $L$ for which the convergents $\tfrac{p_n}{q_n}$ for $z$ of the Hurwitz algorithm satisfy $|z-\tfrac{p_n}{q_n}|< L|q_n|^{-2}$. In his proofs the existence of a finite partition for the Hurwitz algorithm plays a crucial role. Also, the explicit construction of a geometric natural extension of the Hurwitz algorithm from \cite{Ei2018} relies on the existence of a finite partition. Initial investigations into such a construction for the parameter $\balpha = (\frac13, \frac12)$ are done in \cite{mastersthesis}. Based on Theorem~\ref{t:mp} one can now try to obtain such results for $\balpha$-Hurwitz algorithms with rational parameters~$\balpha$.

It would also be interesting to investigate shifted versions of Jeffrey Shallit's algorithm \cite[Section II.3]{Sha79}.

\bibliographystyle{amsplain}
\bibliography{alphaHurwitz}

\end{document}